\documentclass{amsproc}
\usepackage{amsmath}
\usepackage{amsfonts}
\usepackage{amssymb}
\usepackage{hyperref}
\hypersetup{colorlinks,
citecolor=black,
filecolor=black,
linkcolor=black,
urlcolor=black}
\usepackage{graphics,epstopdf}
\usepackage[pdftex]{graphicx}
\usepackage{newlfont}\newlength{\defbaselineskip}
\usepackage[left=1in,right=1in,top=1in,bottom=0.8in,footskip=0.25in]{geometry}
\usepackage{setspace}
\allowdisplaybreaks
\newtheorem{theorem}{Theorem}[section]
\newtheorem{example}{Example}[section]
\newtheorem{lemma}{Lemma}[section]
\newtheorem{definition}[theorem]{Definition}
\newtheorem{remark}{Remark}[section]

\numberwithin{equation}{section}

\usepackage[table]{xcolor}
\newtheorem{corollary}{Corollary}[section]

\usepackage{cite}

\usepackage{fancyhdr}
\usepackage{graphics}
\usepackage{color}
\usepackage{rotating}
\usepackage{fancybox}
\begin{document}
\title{Results on bivariate Sz$\acute{\text{a}}$sz-Mirakjan type operators in  polynomial weight spaces}
\maketitle
\begin{center}
{\bf Rishikesh Yadav$^{1,\dag}$, Ramakanta Meher$^{1,\star}$, Vishnu Narayan Mishra$^{2,\circledast}$}\\
$^{1}$Applied Mathematics and Humanities Department,
Sardar Vallabhbhai National Institute of Technology Surat, Surat-395 007 (Gujarat), India.\\
$^{2}$Department of Mathematics, Indira Gandhi National Tribal University, Lalpur, Amarkantak-484 887, Anuppur, Madhya Pradesh, India\\
\end{center}
\begin{center}
 $^\dag$rishikesh2506@gmail.com,  $^\star$meher\_ramakanta@yahoo.com, $^\circledast$vishnunarayanmishra@gmail.com
\end{center}
\vskip0.5in
\begin{abstract}
In this paper, bivariate Sz$\acute{\text{a}}$sz-Mirakjan type operators are introduced along with the estimation of its approximation properties and its rate of convergence. Furthermore, to check the asymptotic behavior of the said bivariate operators, Voronovskaya type theorem is proved and determined the simultaneous approximation of the operators for first order partial derivative. Finally, the convergence behavior of the bivariate operators are obtained through graphical representation and validated the numerical results with the results of Sz$\acute{\text{a}}$sz-Mirakjan operators for the function of two variables.


\end{abstract}

\subjclass MSC 2010: {41A25, 41A36}

\section{Introduction}\label{se1}
It is easier to compute the function from possibly complected function and it work out on fundamental problems of approximation theory. In 1885, first of all, Weierstrass obtained the significant result of algebraic polynomials for the class of continuous real valued function defined on closed interval. This theorem put a substantial influence in theory of approximation, functional analysis and other parts of mathematics. Picard, Fejer, Landau, de la Vallee, Poussin proved the Weierstrass's theorem.

In the theory of approximation, there were many authors, who introduced various operators to approximate the function easier and nowadays, these process are running on and modifications are taking place for better approximations. Also the various approximations properties, applications have been discussed and applications in other research are being determined. These types of operators are defined on positive  (finite or infinite) interval. Many papers have literature which explains the approximation properties of the corresponding operators for the function of one variable. 

First of all, in 1912, Bernstein \cite{SN1} proposed operators and which are defined on $[0,1]$ for the function of one variable with single parameter $n\in\mathbb{N}$ also in the same direction its generalization is also seen in various papers as Sz$\acute{\text{a}}$sz \cite{OS1}, some modification can be seen in \cite{WM1}, Cheney and Sharma \cite{EC1} generalized the   Bernstein polynomials and further an extension took place in the paper \cite{DS1}. As an extension to infinite interval, many researcher modified and generalized the Sz$\acute{\text{a}}$sz-Mirakjan operators \cite{GM1}. Rempulska and Walczak \cite{LR} developed the rate of convergence and asymptotic behavior is studied by Abel et al. \cite{UAM}. A relation of the Sz$\acute{\text{a}}$sz-Mirakjan operators  by local approximation and local smoothness of function is studied by Jain \cite{GC} in 1972. A further extension is seen in many papers regarding $L_p$ spaces and function belonging to the Integrable spaces.   In \cite{SU}, the authors extended the Jain operators into Kantorovich variant and also to study more approximation properties in another generalization form, Tarabie \cite{ST} studied the Durrmeyer type generalization of the Jain operators. This type of generalization can also be seen in papers \cite{VMP1,PVN1,OA}. Most of the papers are also cited in which regarding convergence properties are studied \cite{SMV,DZ,VT,VNR2,VNR1,TH1,BGJ1,BGJ2,MB2,MB1,TAL,TA1}.

%
%
%

But all above-mentioned papers represent approximation properties for the single variable defined on ($[b,c]$ or $[b,\infty),~b\geq0$).
But to study approximation properties of more than one variable, many researchers generalized  the operators for the function of two variables with the single or two parameters.

Favard \cite{JF} introduced bivariate Sz$\acute{\text{a}}$sz-Mirakjan operators in 1944 and in 1998
Rempulska and Skorupka \cite{LRM3}, considered Sz$\acute{\text{a}}$sz-Mirakjan operators in polynomial weight space for the function of two variables, in addition, they proved the simultaneous approximation theorem as well as studied the Voronowskaya type theorem. The proposed operators for the function of two variables are as follows:
\begin{eqnarray}\label{S}
S_{m,n}(f;x,y)&=& \sum\limits_{k_1=0}^{\infty}\sum\limits_{k_2=0}^{\infty}e^{-mx-ny}\frac{(mx)^{k_1}(ny)^{k_2}}{k_1!k_2!} f\left(\frac{k_1}{m},\frac{k_2}{n} \right),
\end{eqnarray}   
$m,n\in\mathbb{N},~~x,y\geq 0$. 

The above operators were examined by Totik \cite{VT}, for the function of two variables defined and continuous on $[0,\infty)\times[0,\infty)$.
In paper \cite{FDK}, the authors constructed a new type of Sz$\acute{\text{a}}$sz-Mirakjan operators for function of two variables which preserve $x^2+y^2$ and proved that this type of operators have the better rate of convergence in given interval $[0,\infty)\times[0,\infty)$ than classical Sz$\acute{\text{a}}$sz-Mirakjan operators for the function of two variables and also studied the statistical convergence of the sequence of the modified Sz$\acute{\text{a}}$sz-Mirakjan type operators.

To study the behavior of the operators for the function of two variables, many researchers developed their views and modified characteristics of the operators, some are as \cite{MG1,SD1,NI1,MS1,LR}.

%


Motivated by the above works, we introduce the Sz$\acute{\text{a}}$sz-Mirakjan type operators for the function of two variables in the polynomial weight space, which are extension of the defined operators for the single variable by Mishra and Yadav \cite{RYVN3}, 
\begin{eqnarray}\label{RV}
\Hat{\mathcal{R}}_{m,a}(f;x)&=& \sum\limits_{k=0}^{\infty}s_{m}^a(x) ~f\left(\frac{k}{m} \right),
\end{eqnarray}
where
$s_{m}^a(x)=a^{\left(\frac{-x}{-1+a^{\frac{1}{m}}}\right)}\frac{x^{k}(\log{a})^{k}}{(-1+a^{\frac{1}{m}})^{k}k!}$,
 $m\in\mathbb{N}$, $x\in X$ and $a>1$(fixed). They studied approximation properties, uniform convergence of the operators
by using some auxiliary result and also error estimation is given. The convergence of said operators are shown and analyzed by graphics, also in same direction, it is found better rate of convergence than Sz$\acute{\text{a}}$sz-Mirakjan operators by analyzing the graphics.
Voronovskaya-type theorem etc. is determined for asymptotic behavior of the operators. They obtained better rate of convergence than classical Sz$\acute{\text{a}}$sz-Mirakjan operators under some certain conditions such as generalized convexity and also by graphics.

\begin{remark}
Bivariate operators will be defined in next section.
\end{remark}

\section{Construction of the operators and preliminaries}\label{se2}
Through out the paper, it will be used polynomial weight space and for univariate operators we define the weight function as:
\begin{definition}
Consider the space $C_N$ defined using weight $w_N$, $N\in\mathbb{N}$ as follows:
\begin{eqnarray*}
w_0=1, && w_N(x)=(1+x^N)^{-1},~~ (x\in [0,\infty),~N\in\mathbb{N})
\end{eqnarray*}
\begin{eqnarray*}
C_N=\{f\in C[0,\infty): w_Nf~ uniformly~ continuous~ and~ bounded~ on~ [0,\infty)\},
\end{eqnarray*}
endowed with supremum norm
\begin{eqnarray*}
\|f\|_N=\underset{x\geq0}\sup~ w_N|f(x)|.
\end{eqnarray*}
\end{definition}
But for studying the approximation properties of the bivarate operators , it is defined analogy weight function for two parameters. Since the weight function put a influence in the original result of the problem, that's why we use and in fact will get the rate of convergence.

So, next for fixed $N_1,N_2,\in\mathbb{N}\cup\{0\}$, it is defined the weighted function for two parameter $N_1,N_2\in\mathbb{N}$ for function two variables  as:
\begin{eqnarray}
w_{{N_1},{N_2}}=w_{N_1}w_{N_2}, 
\end{eqnarray}
and the polynomial weight space $C_{{N_1},{N_2}}$ is the set of all continuous functions for which  $fw_{{N_1},{N_2}}$ is uniformly continuous and bounded on  $X=[0,\infty)\times[0,\infty)$ endowed with supremum norm as:
\begin{eqnarray}\label{1}
\|f\|_{{N_1},{N_2}}=\underset{x,y\geq0}\sup ~w_{{N_1},{N_2}}(x,y)|f(x,y)|,~~~(t,s)\in X
\end{eqnarray}
Also an inequality can be defined by using mean value theorem:
\begin{eqnarray}
\left|\int\limits_x^t\frac{du}{w_{{N_1},{N_2}}(u,z)}\right|\leq|t-x|\left(\frac{1}{w_{{N_1},{N_2}}(x,z)}+\frac{1}{w_{{N_1},{N_2}}(t,z)} \right).
\end{eqnarray}
Since, by using first mean value theorem, there exist a point $\zeta_{x,t}$ between $x$ and $t$ such that
\begin{eqnarray*}
\left|\int\limits_x^t\frac{du}{w_{{N_1},{N_2}}(u,z)}\right| &=& |t-x| (1+\zeta_{x,t}^{N_1})(1+z^{N_2})\\&\leq & |t-x|(1+x^{N_1}+1+t^{N_1})(1+z^{N_2})\\ &=& \left(\frac{1}{w_{{N_1},{N_2}}(x,z)}+\frac{1}{w_{{N_1},{N_2}}(t,z)} \right).
\end{eqnarray*}
\textbf{Note}: Moreover, 
$C_{N_1,N_2}^d=\{ f\in C_{N_1,N_2}:\frac{\partial^c }{\partial^ix \partial^{c-i}y }\in C_{N_1,N_2},i=1\dots c, c=1\dots d\}$.
\begin{definition}
Let $f\in C_{N_1,N_2}(X)$, the modulus of continuity is defined by 
\begin{eqnarray}\label{D1}
\omega(f,C_{N_1,N_2};\delta_1,\delta_2)=\underset{\underset{0\leq h_2\leq\delta_2}{0\leq h_1\leq\delta_1}}\sup\| \Delta_{h_1,h_2}f\|_{N_1,N_2},~~~\delta_1,\delta_2>0,
\end{eqnarray}
where $\Delta_{h_1,h_2}=f(x+h_1,y+h_2)-f(x,y)$.
\end{definition}
\begin{definition}\label{def1}
For a given $N_1,N_2\in\mathbb{N}\cup\{0\}$ and $\alpha>0, \beta\geq 1$, we have 
\begin{eqnarray*}
Lip\left(C_{N_1,N_2};\alpha,\beta \right)=\left\{f\in C_{N_1,N_2}:\omega(f,C_{N_1,N_2};\delta_1,\delta_2)=O\left(\delta_1^{\alpha}+\delta_2^{\beta}\right)~\text{as}~\delta_1,\delta_2\to 0^+  \right\}.
\end{eqnarray*} 
\begin{remark}
In particular, here  
\begin{eqnarray*}
C_{N_1,N_2}^1=\left\{f\in C_{N_1,N_2}:\frac{\partial f}{\partial x}, \frac{\partial f}{\partial y}\in C_{N_1,N_2} \right\}.
\end{eqnarray*}
\end{remark}
\end{definition}


Now, we define bivariate Sz$\acute{\text{a}}$sz-Mirakjan type operators for function of two variables $f(x,y)\in C_{N_1,N_2}$, $N_1,N_2>0$ with  $x,y\in X$ which are bivariate extension of the define operators (\ref{RV}) and thus

\begin{eqnarray}\label{R1}
\Hat{Y}_{m,n,a}(f;x,y)&=& \sum\limits_{k_1=0}^{\infty}\sum\limits_{k_2=0}^{\infty}s_{m,n}^a(x,y) ~f\left(\frac{k_1}{m},\frac{k_2}{n} \right),
\end{eqnarray}
where
$s_{m,n}^a(x,y)=a^{\left(\frac{-x}{-1+a^{\frac{1}{m}}}\right)}a^{\left(\frac{-y}{-1+a^{\frac{1}{n}}}\right)}\frac{x^{k_1}y^{k_2}(\log{a})^{k_1+k_2}}{(-1+a^{\frac{1}{m}})^{k_1}(-1+a^{\frac{1}{n}})^{k_2}k_1!k_2!}$,
 $m,n\in\mathbb{N}$, $(x,y)\in X$ and $a>1$(fixed). 
 

 
 \begin{remark}
 For all $x,y\in X$, we have 
 \begin{eqnarray}
\sum\limits_{k_1=0}^{\infty}\sum\limits_{k_2=0}^{\infty}s_{m,n}^a(x,y)=1.
 \end{eqnarray}
 \end{remark}
\begin{remark}
If $f(x,y)=f_1(x)f_2(y)$, where $f\in C_{N_1,N_2}(X)$ then 
\begin{eqnarray}\label{2}
\Hat{Y}_{m,n,a}(f(t,s);x,y)=\Hat{\mathcal{R}}_{m,a}(f_1(t);x)\times \Hat{\mathcal{R}}_{n,a}(f_2(s);x),
\end{eqnarray}
where $\Hat{\mathcal{R}}_{m,a}(f_1(t);x)$ and $\Hat{\mathcal{R}}_{n,a}(f_2(s);x)$ can be obtained by defined operators (\ref{RV}).
\end{remark}

Now we shall give some results which are used to prove our main theorems.

\begin{lemma}\label{L1} Let $x,y\geq 0$ and for each $m,n\in \mathbb{N}$, it holds
\begin{eqnarray}
\Hat{Y}_{m,n,a}(e_{10};x,y)&=& \frac{x \log (a)}{m \left(a^{\frac{1}{m}}-1\right)},\\
\Hat{Y}_{m,n,a}(e_{01};x,y)&=& \frac{y \log (a)}{n \left(a^{\frac{1}{n}}-1\right)},\\
\Hat{Y}_{m,n,a}(e_{20};x,y)&=& \frac{x \log (a) \left(a^{\frac{1}{m}}+x \log (a)-1\right)}{m^2 \left(a^{\frac{1}{m}}-1\right)^2},\\
\Hat{Y}_{m,n,a}(e_{02};x,y)&=& \frac{y \log (a) \left(a^{\frac{1}{n}}+y \log (a)-1\right)}{n^2 \left(a^{\frac{1}{n}}-1\right)^2}.
\end{eqnarray}
\end{lemma}

\begin{lemma}\label{L2}
For every $x,y\in X=[0,\infty)\times[0,\infty)$ and $m,n\in\mathbb{N}$, the following results hold:

\begin{eqnarray*}
1.~\Hat{Y}_{m,n,a}((t-x);x,y)&=&-\frac{x \left(m a^{\frac{1}{m}}-\log (a)-m\right)}{m \left(a^{\frac{1}{m}}-1\right)}\\
2.~\Hat{Y}_{m,n,a}((s-y);x,y)&=& -\frac{y \left(n a^{\frac{1}{n}}-\log (a)-n\right)}{n \left(a^{\frac{1}{n}}-1\right)}\\
3.~\Hat{Y}_{m,n,a}((t-x)^2;x,y)&=&\frac{x \left(m^2 x \left(a^{\frac{1}{m}}-1\right)^2-\left(a^{\frac{1}{m}}-1\right) \log (a) (2 m x-1)+x \log ^2(a)\right)}{m^2 \left(a^{\frac{1}{m}}-1\right)^2}\\
4.~\Hat{Y}_{m,n,a}((s-y)^2;x,y)&=&\frac{y \left(n^2 y \left(a^{\frac{1}{n}}-1\right)^2-\left(a^{\frac{1}{n}}-1\right) \log (a) (2 n y-1)+y \log ^2(a)\right)}{n^2 \left(a^{\frac{1}{n}}-1\right)^2}\\
5.~\Hat{Y}_{m,n,a}((t-x)^4;x,y)&=&\frac{x}{m^4 \left(a^{\frac{1}{m}}-1\right)^4}\Bigg\{m^4 x^3 \left(a^{\frac{1}{m}}-1\right)^4-\left(a^{\frac{1}{m}}-1\right)^3 (-1 + 4 m x - 6 m^2 x^2 + 4 m^3 x^3) \log{a}\\&& +\left(a^{\frac{1}{m}}-1\right)^2 x (7 - 12 m x + 6 m^2 x^2) (\log{a})^2 \\&&
-2 \left(a^{\frac{1}{m}}-1\right) x^2 (-3 + 2 m x) (\log{a})^3 + x^3 (\log{a})^4\Bigg\} \\
6.~\Hat{Y}_{m,n,a}((s-y)^4;x,y)&=&\frac{y}{n^4 \left(a^{\frac{1}{n}}-1\right)^4}\Bigg\{n^4 y^3 \left(a^{\frac{1}{n}}-1\right)^4-\left(a^{\frac{1}{n}}-1\right)^3 (-1 + 4 n y - 6 n^2 y^2 + 4 n^3 y^3) \log{a}\\&& +\left(a^{\frac{1}{n}}-1\right)^2 y (7 - 12 n y + 6 n^2 y^2) (\log{a})^2 \\&&
-2 \left(a^{\frac{1}{n}}-1\right) y^2 (-3 + 2 n y) (\log{a})^3 + y^3 (\log{a})^4\Bigg\}.
\end{eqnarray*}
\end{lemma}

\begin{proof}
For all $x,y\geq 0$, 
\begin{eqnarray*}
1.~ \Hat{Y}_{m,n,a}((t-x);x,y)&=& \sum\limits_{k_1=0}^{\infty}\sum\limits_{k_2=0}^{\infty}s_{m,n}^a(x,y)\left(\frac{k_1}{m} -x\right)\\
&=& a^{\left(\frac{-x}{-1+a^{\frac{1}{m}}}\right)} \sum\limits_{k_1=0}^{\infty} \frac{x^{k}(\log{a})^{k_1}}{(-1+a^{\frac{1}{m}})^{k_1}k_1!}\left(\frac{k_1}{m}-x \right)\\
&=&-\frac{x \left(m a^{\frac{1}{m}}-\log (a)-m\right)}{m \left(a^{\frac{1}{m}}-1\right)}
\end{eqnarray*}
\begin{eqnarray*}
3.~\Hat{Y}_{m,n,a}((t-x)^2;x,y)&=& \sum\limits_{k_1=0}^{\infty}\sum\limits_{k_2=0}^{\infty}s_{m,n}^a(x,y)\left(\left( \frac{k_1}{m}\right)^2 +x^2-\frac{2xk_1}{m}\right)\\
&=& a^{\left(\frac{-x}{-1+a^{\frac{1}{m}}}\right)} \sum\limits_{k_1=0}^{\infty} \frac{x^{k}(\log{a})^{k_1}}{(-1+a^{\frac{1}{m}})^{k_1}k_1!}\left( \frac{k_1}{m}\right)^2+x^2- 2x \frac{x \log (a)}{m \left(a^{\frac{1}{m}}-1\right)}\\
&=& \frac{x \log (a) \left(a^{\frac{1}{m}}+x \log (a)-1\right)}{m^2 \left(a^{\frac{1}{m}}-1\right)^2}+x^2\left(1-\frac{2 \log (a)}{m \left(a^{\frac{1}{m}}-1\right)} \right)\\
&=& \frac{x \left(m^2 x \left(a^{\frac{1}{m}}-1\right)^2-\left(a^{\frac{1}{m}}-1\right) \log (a) (2 m x-1)+x \log ^2(a)\right)}{m^2 \left(a^{\frac{1}{m}}-1\right)^2},
\end{eqnarray*}
similarly, other equalities can be proved by the same process.
\end{proof}

\begin{lemma} \label{L23}
For all $x\geq 0$, we get
\begin{eqnarray}
\Hat{Y}_{m,n,a}((t-x)^2;x,y)&\leq & \frac{x(1+x)}{m}=\delta_m^{'2}(x)~~(\text{say}).\\
\Hat{Y}_{m,n,a}((s-y)^2;x,y)&\leq & \frac{y(1+y)}{n}=\delta_m^{'2}(y)~~(\text{say}).
\end{eqnarray}
\end{lemma}
\begin{proof} For $m,n\in\mathbb{N}$ and $x,y\geq 0$, we reach to
\begin{eqnarray}
\nonumber \Hat{Y}_{m,n,a}((t-x)^2;x,y)&= & x\left(x-\frac{2x\log{a}}{m\left(a^{\frac{1}{m}}-1\right)}+\frac{\log{a}}{{m^2\left(a^{\frac{1}{m}}-1\right)}}+\frac{x(\log{a})^2}{{m^2\left(a^{\frac{1}{m}}-1\right)^2}} \right)\\\nonumber
&\leq & x \left(x\left(\frac{\log{a}}{m\left(a^{\frac{1}{m}}-1\right)}-1 \right)^2+\frac{1}{m}\right)\\
&\leq & \frac{x(1+x)}{m}.
\end{eqnarray}
Similarly, next proof can be completed.
\end{proof}

\begin{lemma} For all natural numbers $m,n$, we get the following limits
\begin{enumerate}
\item{} $\underset{m\to\infty}\lim m\Hat{Y}_{m,n,a}((t-x);x,y)=-\frac{1}{2} x \log (a)$
\item{} $\underset{m\to\infty}\lim m\Hat{Y}_{m,n,a}((t-x)^2;x,y)=x$
\item{} $\underset{m\to\infty}\lim \Hat{Y}_{m,n,a}((t-x)^4;x,y)=0$
\item{} $\underset{m\to\infty}\lim m\Hat{Y}_{m,n,a}((t-x)^4;x,y)=0$
\item{} $\underset{m\to\infty}\lim m^2\Hat{Y}_{m,n,a}((t-x)^4;x,y)=3x^2$ 
\end{enumerate}
\end{lemma}
\begin{proof}
By above Lemma \ref{L2}, we obtain
\begin{eqnarray*}
1.~\underset{m\to\infty}\lim m\Hat{Y}_{m,n,a}((t-x);x,y)&=& \underset{m\to\infty}\lim \frac{-x \left(a^{\frac{1}{m}}-\frac{\log{a}}{m}-1\right)}{ \frac{\left(a^{\frac{1}{m}}-1\right)}{m}},
\end{eqnarray*}
using L. Hospital's rule as undetermined form, we get the required result.\\
Note: Here to prove all above parts, we shall use L. Hospital rule.
\end{proof}

%
%

\textbf{Note:} We shall use  $\mathcal{M}_r$, $r=1,2,3,\cdots$ with the suitable choice as for constant.
\begin{lemma}\label{L10}
For every $x\in [0,\infty)$ and $N_1\in \mathbb{N}\cup \{0\}$, there exist a positive constant $\mathcal{M}_r$, $r=1,2$, such that
\begin{enumerate}
\item{} $w_{N_1}(x) \Hat{Y}_{m,n,a}\left(\frac{1}{w_{N_1}};x\right) \leq  \mathcal{M}_1(N_1),$
\item{} $w_{N_1}(x) \Hat{Y}_{m,n,a}\left(\frac{(t-x)^2}{w_{N_1}};x\right)\leq \mathcal{M}_2(N_1)\delta_m^{'2},~~m\in\mathbb{N}.$
\end{enumerate}

\end{lemma}

The main objective of this paper, to investigate the approximation properties in polynomial weight space, moreover Voronovskaya type theorem and simultaneous approximation are studied. Also rate of convergence is discussed.

There are following sections by which it can be characterized to our work and related properties. Section (\ref{se3}) deals with the approximation properties in polynomial weight space. Section (\ref{se4}), consists the Voronovskaya type theorem. To check the convergence of the derivative of the  bivriate operators to the function of two variables, we prove simultaneous approximation of the operators in section (\ref{se5}). Section (\ref{se6}) represent the tabular and graphical representation.\\


\section{Approximation properties of the bivariate operators}\label{se3}
In this section, the approximation properties will be studied with the help of polynomial weight space.  
\begin{theorem}
Let $f\in C_{N_1,N_2}^1$, for all $x,y\in X$ and $\exists$ a positive constant $\mathcal{M}_5$ such that
\begin{eqnarray}\label{T1}
w_{N_1,N_2}|\Hat{Y}_{m,n,a}(f;x,y)-f(x,y)|&\leq & \mathcal{M}_5(N_1,N_2)\Bigg\{\|f'_x\|_{N_1,N_2}\delta'_m(x)+\|f'_y\|_{N_1,N_2}\delta'_n(y) \Bigg\},
\end{eqnarray} 
where $\delta_m'(x)=\sqrt{\frac{x(1+x)}{m}}$ and $\delta_n'(y)=\sqrt{\frac{y(1+y)}{n}}$.
\end{theorem}
\begin{proof}
Since $(x,y)\in X$ and let $t,s\geq0$ then we have
\begin{eqnarray*}
f(t,s)-f(x,y)=\int\limits_x^t f'_u(u,y)~du+\int\limits_y^s f'_v(x,v)~dv,
\end{eqnarray*}
on applying $\Hat{Y}_{m,n,a}(f;x,y)$ on both side we get,
\begin{eqnarray}\label{4}
\Hat{Y}_{m,n,a}(f;x,y)-f(x,y)=\Hat{Y}_{m,n,a}\left(\int\limits_x^t f'_u(u,y)~du;x,y\right)+\Hat{Y}_{m,n,a}\left(\int\limits_y^s f'_v(x,v)~dv;x,y \right),
\end{eqnarray}
using \ref{1}, we get
\begin{eqnarray*}
\left|\int\limits_x^t f'_u(u,y)~du\right|\leq \|f'_x\|_{N_1,N_2}\left|\int\limits_x^t \frac{1}{w_{N_1,N_2}(u,s)}~du\right|\leq \|f'_x\|_{N_1,N_2}\left(\frac{1}{w_{N_1,N_2(t,s)}}+\frac{1}{w_{N_1,N_2}(x,s)} \right)|t-x|,
\end{eqnarray*}
and same as 
\begin{eqnarray*}
\left|\int\limits_x^t f'_u(x,v)~dv\right|\leq \|f'_y\|_{N_1,N_2}\left(\frac{1}{w_{N_1,N_2(t,s)}}+\frac{1}{w_{N_1,N_2}(t,y)} \right)|s-y|
\end{eqnarray*}
From (\ref{2}) and using these inequalities, we get
\begin{eqnarray}\label{EI4}
\nonumber  w_{N_1,N_2}(x,y)\left|\Hat{Y}_{m,n,a}\left(\int\limits_x^t f'_u(u,y)~du;x,y\right) \right| &\leq & w_{N_1,N_2}(x,y)\Hat{Y}_{m,n,a}\left(\left|\int\limits_x^t f'_u(u,y)~du \right| ;x,y\right)\nonumber\\
&\leq & \|f'_x\|_{N_1,N_2}w_{N_1,N_2} \Bigg\{  \Hat{Y}_{m,n,a}\left(\frac{|t-x|}{w_{N_1,N_2}(t,s)};x,y \right)\nonumber\\&&
+\Hat{Y}_{m,n,a}\left(\frac{|t-x|}{w_{N_1,N_2}(x,s)};x,y \right) \Bigg\}\nonumber\\
&=& \|f'_x\|_{N_1,N_2}w_{N_2}\Hat{Y}_{n,a}\left(\frac{1}{w_{N_2}(s)};y\right)\Bigg\{w_{N_1}  \Hat{Y}_{m,a}\left(\frac{|t-x|}{w_{N_1}(t)};x \right)\nonumber\\&&
+Y_{m,a}^*\left(|t-x|;x \right) \Bigg\},
\end{eqnarray}
and same way
\begin{eqnarray}\label{IE5}
\nonumber w_{N_1,N_2}(x,y)\left|\Hat{Y}_{m,n,a}\left(\int\limits_x^t f'_u(x,v)~dv;x,y\right) \right| &\leq &  \|f'_y\|_{N_1,N_2}w_{N_1}\Hat{Y}_{m,a}\left(\frac{1}{w_{N_1}(t)};x\right)\Bigg\{w_{N_2}  \Hat{Y}_{n,a}\left(\frac{|s-y|}{w_{N_2}(s)};y \right)\\&&
+Y_{n,a}^*\left(|s-y|;y \right) \Bigg\},
\end{eqnarray}
using Lemma \ref{L2}, and by H$\ddot{\text{o}}$lder inequality, we obtain
\begin{eqnarray}\label{IE6}
Y_{m,a}^*(|t-x|;x)\leq \{Y_{m,a}^*((t-x)^2;x)\times Y_{m,a}^*(1;x) \}^{\frac{1}{2}}\leq \delta'_m,
\end{eqnarray}
and 
\begin{eqnarray}\label{EI7}
w_{N_1}(x)\Hat{Y}_{m,a}\left(\frac{|t-x|}{w_{N_1(t)}};x\right)&\leq &\Bigg\{\Hat{Y}_{m,a}\left(\frac{(t-x)^2}{w_{N_1(t)}};x\right)\times \Hat{Y}_{m,a}\left(\frac{1}{w_{N_1}};x\right) \Bigg\}^{\frac{1}{2}}\leq \mathcal{M}_3(N_1)\delta'_m,
\end{eqnarray}
analogy 
\begin{eqnarray}\label{EI8}
\Hat{Y}_{n,a}(|s-y|;y)\leq \delta'_n,
\end{eqnarray}
and 
\begin{eqnarray}\label{EI9}
w_{N_2}(y)\Hat{Y}_{n,a}\left(\frac{|s-y|}{w_{N_2(s)}};y\right) \leq \mathcal{M}_4(N_2)\delta'_n.
\end{eqnarray}
Hence from (\ref{4}-\ref{EI9}), we have

\begin{eqnarray}
w_{N_1,N_2}|\Hat{Y}_{m,n,a}(f;x,y)-f(x,y)|&\leq & \mathcal{M}_5(N_1,N_2)\Bigg\{\|f'_x\|_{N_1,N_2}\delta'_m(x)+\|f'_y\|_{N_1,N_2}\delta'_n(y) \Bigg\}.
\end{eqnarray}
Thus the proof is completed.
\end{proof}


Next theorem represent a pioneer result, associated with  Stecklov function $f\in C_{N_1,N_2}(X)$.

\begin{theorem}\label{T3}
For some $N_1,N_2\in \mathbb{N}\cup \{0\}$, we suppose that $f\in C_{N_1,N_2}(X)$ then $\exists$ a positive constant $\mathcal{M}_9$, we  obtain
\begin{eqnarray}
w_{N_1,N_2}\|\Hat{Y}_{m,n,a}(f;x,y)-f(x,y)\|_{N_1,N_2} &\leq & \mathcal{M}_{9}(N_1,N_2)\omega(f,C_{N_1,N_2};\delta'_m,\delta'_n).
\end{eqnarray}
\end{theorem}
\begin{proof}
Now we prove the above theorem with the help of Steklov formula. So consider $f_{h_1,h_2}$ be the Steklov function belonging to $C_{N_1,N_2}$ and defined by 
\begin{eqnarray}
f_{h_1,h_2}(x,y)=\frac{1}{h_1h_2}\int\limits_0^{h_1}\int\limits_0^{h_2} f(x+\zeta,y+\eta)~d\zeta~d\eta,~~~~\zeta,\eta>0
\end{eqnarray} 
so we have,
\begin{eqnarray}
f_{h_1,h_2}(x,y)-f(x,y)=\frac{1}{h_1h_2}\int\limits_0^{h_1}d\zeta\int\limits_0^{h_2} \Delta_{\zeta,\eta}f(x,y)d\eta,
\end{eqnarray}
Differentiate partially with respect to $x$, one can obtains
\begin{eqnarray}
\frac{\partial}{\partial x} f_{h_1,h_2}(x,y)&=& \frac{1}{h_1h_2}\int\limits_0^{h_1}d\zeta\int\limits_0^{h_2} \Delta_{h_1,0}f(x,y+\eta)d\eta\\&=&  \frac{1}{h_1h_2}\int\limits_0^{h_2}(\Delta_{h_1,\eta}f(x,y)-\Delta_{0,\eta}f(x,y))~d\eta
\end{eqnarray}
Similarly, on differentiating partially with respect to $y$, we get
\begin{eqnarray}
\frac{\partial}{\partial x} f_{h_1,h_2}(x,y)&=& \frac{1}{h_1h_2}\int\limits_0^{h_1}(\Delta_{\zeta,h_2}f(x,y)-\Delta_{\zeta,0}f(x,y))~d\zeta,
\end{eqnarray}
using (\ref{1}) and (\ref{D1}), we obtain
\begin{eqnarray}\label{E17}
\|f_{h_1,h_2}-f \|_{N_1,N_2} \leq \omega(f,C_{N_1,N_2};h_1,h_2)
\end{eqnarray}
\begin{eqnarray}\label{E18}
\left\|\frac{\partial f_{h_1,h_2}}{\partial x} \right\|_{N_1,N_2} \leq \frac{2}{h_1} \omega(f,C_{N_1,N_2};h_1,h_2),
\end{eqnarray}
similarly,
\begin{eqnarray}\label{E19}
\left\|\frac{\partial f_{h_1,h_2}}{\partial x} \right\|_{N_1,N_2} \leq \frac{2}{h_2} \omega(f,C_{N_1,N_2};h_1,h_2),
\end{eqnarray}
so we have 
\begin{eqnarray}\label{E20}
\nonumber w_{N_1,N_2}|\Hat{Y}_{m,n,a}(f;x,y)-f(x,y)|&\leq & w_{N_1,N_2} \Bigg\{\Hat{Y}_{m,n,a}(f(t,s)-f_{h_1,h_2}(t,s);x,y)\nonumber \\&&
+|\Hat{Y}_{m,n,a}(f_{h_1,h_2}(t,s);x,y)-f_{h_1,h_2}(x,y)|+ |f_{h_1,h_2}(x,y)-f(x,y)|\Bigg\}\nonumber \\
&=& J_1+J_2+J_3
\end{eqnarray}
Now taking supremum norm on $C_{N_1,N_2}$ and by \ref{E17} we can get with the help of Lemma [5] as:
 \begin{eqnarray}
 \nonumber J_1=\|\Hat{Y}_{m,n,a}(f(t,s)-f_{h_1,h_2}(t,s);x,y)\|_{N_1,N_2}&\leq & M_7(N_1,N_2) \|f-f_{h_1,h_2}\|_{N_1,N_2} \\
 &\leq & \mathcal{M}_6(N_1,N_2) \omega(f,C_{N_1,N_2};h_1,h_2),
 \end{eqnarray}
also 
\begin{eqnarray}
J_2\leq \omega(f,C_{N_1,N_2};h_1,h_2).
\end{eqnarray}
By previous theorem \ref{T1} and with the help of \ref{E18}-\ref{E19},  it is follows that 
 \begin{eqnarray*}
J_2 &\leq & \mathcal{M}_7(N_1,N_2)\Bigg\{\left\|\frac{\partial f_{h_1,h_2}}{\partial x} \right\|_{N_1,N_2}\delta'_m +\left\|\frac{\partial f_{h_1,h_2}}{\partial x} \right\|_{N_1,N_2}\delta'_n \Bigg\} \\
&\leq & 2\mathcal{M}_8(N_1,N_2)\omega(f,C_{N_1,N_2};h_1,h_2)\Bigg\{\frac{\delta'_m}{h_1}+\frac{\delta'_n}{h_2} \Bigg\},
\end{eqnarray*}
thus we reach to our required result with  $J_1,J_2,J_3$  by using in \ref{E20}, we get 
\begin{eqnarray}
w_{N_1,N_2}\|\Hat{Y}_{m,n,a}(f;x,y)-f(x,y)\|_{N_1,N_2} &\leq & M_9(N_1,N_2)\omega(f,C_{N_1,N_2};h_1,h_2)\Bigg\{1+ \frac{\delta'_m}{h_1}+\frac{\delta'_n}{h_2}\Bigg\},
\end{eqnarray}
to prove the theorem, we consider $h_1=\delta'_m$ and $h_2=\delta'_n$ , where $h_1,h_2>0$, then we get 
\begin{eqnarray}
w_{N_1,N_2}\|\Hat{Y}_{m,n,a}(f;x,y)-f(x,y)\|_{N_1,N_2} &\leq & \mathcal{M}_{9}(N_1,N_2)\omega(f,C_{N_1,N_2};\delta'_m,\delta'_n).
\end{eqnarray}
\end{proof}
\begin{remark}
The property of modulus of continuity holds, $\omega (f,C_{N_1,N_2};\delta'_m,\delta'_n)\to 0$ as $\delta'_m,\delta'_n\to 0^{+}$ for $m,n\to\infty$.
\end{remark}
The consequence of the above Theorem \ref{T3}, is as follows:
\begin{theorem}\label{TH3}
Let $m,n\in\mathbb{N}$ and the operators defined by \ref{R1}. Let $f\in C_{N_1,N_2}$, then the pointwise convergence take place 
\begin{eqnarray}\label{E21}
\underset{m,n\to\infty}\lim \Hat{Y}_{m,n,a}(f;x,y)=f(x,y),
\end{eqnarray} 
but for any compact set $X_1=[0,l_1]\times[0,l_2]$, $l_1,l_2>0$, then (\ref{E21}) holds uniformly on domain $X_1$, i.e. uniformely on evry rectangular $0\leq x\leq l_1, 0\leq y\leq l_2$.
\end{theorem}
Also, one more consequence of the above Theorem \ref{T3} is as follows
\begin{corollary}
If $f\in Lip\left(C_{N_1,N_2};\alpha,\beta \right)$ for some $N_1,N_2\in \mathbb{N}\cup\{0\}$ with $\alpha>0, \beta\geq 1$ then for every  $x,y\in X$, it holds 
\begin{eqnarray*}
_{N_1,N_2}\|\Hat{Y}_{m,n,a}(f;x,y)-f(x,y)\|_{N_1,N_2}=O\left((\delta'_m)^{\alpha}+ (\delta'_n)^{\beta} \right),~~\text{as}~m,n,\to\infty.
\end{eqnarray*}
\end{corollary}
\section{Voronovskaya type theorem}\label{se4}

\begin{theorem}
Let $f\in C_{N_1,N_2}^2(X)$,  for every $(x,y)\in X$, then the  Voronovskaya  type formula is as
\begin{eqnarray*}
\underset{m\to\infty}\lim m\{\Hat{Y}_{m,n,a}(f;x,y)-f(x,y)\}=\frac{-x\log{a}}{2}f_x(x,y)+\frac{-y\log{a}}{2}f_y(x,y)+\frac{x}{2}f_{xx}(x,y)+\frac{y}{2}f_{yy}(x,y).
\end{eqnarray*} 
\end{theorem}
\begin{proof}
Let $(x_0,y_0)\in X$ be a fixed point, then by Taylor formula
\begin{eqnarray}\label{E1}
\nonumber f(t,s)&=& f(x_0,y_0)+f_x(x_0,y_0)(t-x_0)+f_y(x_0,y_0)(s-y_0)\\\nonumber&&+\frac{1}{2}\{ f_{xx}(x_0,y_0)(t-x_0)^2+ 2f_{xy}(x_0,y_0)(t-x_0)(s-y_0)+ f_{yy}(x_0,y_0)(s-y_0)^2\}\\ &&+\xi(t,s;x_0,y_0)\sqrt{(t-x_0)^4+(s-y_0)^4},
\end{eqnarray}

where, $(t,s)\in X$ and $\xi(t,s)=\xi(t,s;x_0,y_0)\in C_{N_1,N_2}(X)$ and also $\underset{(t,s)\to(x_0,y_0)}\lim\xi(t,s)=0$.
Now applying the bivariate operators (\ref{R1}) on above equation (\ref{E1}), one has

\begin{eqnarray}\label{E2}
\nonumber \Hat{Y}_{m,m,a}(f;x_0,y_0)-f(x_0,y_0)&=& f_x(x_0,y_0)\Hat{Y}_{m,m,a}((t-x_0);x_0)+f_y(x_0,y_0)\Hat{Y}_{m,m,a}((s-y_0);y_0)\\\nonumber&&+\frac{1}{2}\{ f_{xx}(x_0,y_0)\Hat{Y}_{m,m,a}((t-x_0)^2;x_0)+ 2f_{xy}(x_0,y_0)\Hat{Y}_{m,m,a}((t-x_0);x_0)\Hat{Y}_{m,m,a}((s-y_0);y_0)\\ &&+ f_{yy}(x_0,y_0)\Hat{Y}_{m,m,a}((s-y_0)^2;y_0)\}+\Hat{Y}_{m,m,a}(\xi(t,s);x_0,y_0)\sqrt{(t-x_0)^4+(s-y_0)^4};x_0,y_0).
\end{eqnarray}
using above lemma (\ref{L23}), we have

\begin{eqnarray}\label{E3}
\underset{m\to\infty}\lim{m \Hat{Y}_{m,m,a}((t-x_0);x_0))}=-\frac{1}{2} x_0 \log (a), \underset{m\to\infty}\lim{m(\Hat{Y}_{m,m,a}((s-y_0);y_0))}=-\frac{1}{2} y_0 \log (a)
\end{eqnarray}
\begin{eqnarray}\label{E4}
\underset{m\to\infty}\lim{m(\Hat{Y}_{m,m,a}((t-x_0)^2;x_0))}=x_0, \underset{m\to\infty}\lim{m(\Hat{Y}_{m,m,a}((s-y_0)^2;y_0))}=y_0
\end{eqnarray}
Now applying H$\ddot{\text{o}}$lder inequality, we have
\begin{eqnarray}\label{E5}
\nonumber|\Hat{Y}_{m,m,a}(\xi(t,s;x_0,y_0)\sqrt{(t-x_0)^4(s-y_0)^4};x_0,y_0)| &\leq & \{\Hat{Y}_{m,m,a}(\xi^2(t,s);x_0,y_0)\}^{\frac{1}{2}}\\&&\times \{\Hat{Y}_{m,m,a}((t-x_0)^4;x_0)+\Hat{Y}_{n,a}((s-y_0)^4;y_0)\}^{\frac{1}{2}},
\end{eqnarray}
using the property of $\xi$ and Theorem \ref{TH3}, one has as:

\begin{eqnarray}
\underset{n\to\infty}\lim \Hat{Y}_{m,m,a}(\xi^2(t,s);x_0,y_0))=\xi^2(x_0,y_0)=0
\end{eqnarray}
in next one step, we use Lemma \ref{L23} and reach to
\begin{eqnarray}\label{E6}
\underset{n\to\infty}{\lim} m\Hat{Y}_{m,m,a}(\xi(t,s);x_0,y_0)\sqrt{(t-x_0)^4+(s-y_0)^4};x_0,y_0)=0.
\end{eqnarray}
Using (\ref{E2},\ref{E3},\ref{E4},\ref{E5},\ref{E6}), we reproduce \ref{E1}
\begin{eqnarray*}
\underset{m\to\infty}\lim m\{\Hat{Y}_{m,m,a}(f;x,y)-f(x,y)\}=\frac{-x\log{a}}{2}f_x(x,y)+\frac{-y\log{a}}{2}f_y(x,y)+\frac{x}{2}f_{xx}(x,y)+\frac{y}{2}f_{yy}(x,y).
\end{eqnarray*} 
Thus the proof is completed.
\end{proof}

\section{Simultaneous approximation}\label{se5}
\begin{theorem}
For $f\in C_{N_1,N_2}^1$, for all $m\in \mathbb{N}$, then for each $(x,y)\in X$, we approach to:
\begin{eqnarray}
\underset{m\to\infty}{\lim}\frac{\partial }{\partial x}(\Hat{Y}_{m,m,a}(f;x,y))=\frac{\partial f(x,y)}{\partial x},
\end{eqnarray}

\begin{eqnarray}
\underset{m\to\infty}{\lim}\frac{\partial }{\partial y}(\Hat{Y}_{m,m,a}(f;x,y))=\frac{\partial f(x,y)}{\partial y}.
\end{eqnarray}
\end{theorem}

\begin{proof}
By using Taylor's theorem,
\begin{eqnarray}\label{E7}
f(t,s)=f(x,y)+f_x(x,y)(t-x)+f_y(x,y)(s-y)+\xi(t,s;x,y)\sqrt{(t-x)^2+(s-y)^2}.
\end{eqnarray}

%
Since, the partial derivative of the bivariate operators \ref{R1}, one can derive as:
\begin{eqnarray}\label{E8}
\frac{\partial }{\partial x}(\Hat{Y}_{m,m,a}(f;x,y)= -\frac{ \log (a)}{ \left(a^{\frac{1}{m}}-1\right)}\Hat{Y}_{m,m,a}(f(t,s);x,y)+\frac{m}{x}~\Hat{Y}_{m,m,a}(tf(t,s);x,y).
\end{eqnarray}
Applying the bivriate operators $\Hat{Y}_{m,m,a}(f;x,y)$ \ref{R1} and using \ref{E8}, we have

\begin{eqnarray}\label{E9}
\nonumber\frac{\partial }{\partial x}(\Hat{Y}_{m,m,a}(f;x,y)&=& -\frac{ \log (a)}{ \left(a^{\frac{1}{m}}-1\right)} \Bigg(f(x,y)+f_x(x,y)\Hat{Y}_{m,m,a}((t-x);x,y)+f_y(x,y)\Hat{Y}_{m,m,a}((s-y);x,y)\\\nonumber
&&+\Hat{Y}_{m,m,a}(\xi(t,s;x,y)\sqrt{(t-x)^2+(s-y)^2};x,y) \Bigg)+\frac{m}{x}\Bigg(f(x,y)\Hat{Y}_{m,m,a}(t;x,y)\\\nonumber
&&+f_x(x,y)\Hat{Y}_{m,m,a}(t(t-x);x,y)+f_y(x,y)\Hat{Y}_{m,m,a}(t(s-y);x,y)\\
&&+\Hat{Y}_{m,m,a}(t~\xi(t,s;x,y)\sqrt{(t-x)^2+(s-y)^2};x,y) \Bigg),
\end{eqnarray}
with the help of Lemma \ref{L1} and \ref{L2}, we get
\begin{eqnarray*}
\Hat{Y}_{m,m,a}(t(s-y);x,y)&=& \Hat{Y}_{m,m,a}(t;x,y)\times\Hat{Y}_{m,m,a}((s-y);x,y)\\&=&\frac{-xy \log (a) \left(n a^{\frac{1}{m}}-\log (a)-m\right)}{m^2 \left(a^{\frac{1}{m}}-1\right)^2 },
\end{eqnarray*}
so, we have
\begin{eqnarray}\label{E10}
\nonumber\frac{\partial }{\partial x}(\Hat{Y}_{m,m,a}(f;x,y)&=& f_x(x,y)\left( \frac{\log (a)}{m \left(a^{\frac{1}{m}}-1\right)} \right)+\frac{m}{x}~ \Hat{Y}_{m,m,a}\Bigg( \left(t- \frac{x\log(a)}{ \left( a^{\frac{1}{m}}-1\right)m} \right) \\&& \times ~\xi(t,s;x,y)\sqrt{(t-x)^2+(s-y)^2};x,y\Bigg),
\end{eqnarray}
now 
\begin{eqnarray}\label{E11}
\nonumber&&\Hat{Y}_{m,m,a}\left( \left(t- \frac{x\log(a)}{ \left( a^{\frac{1}{m}}-1\right)m} \right) \xi(t,s;x,y)\sqrt{(t-x)^2+(s-y)^2};x,y\right)\\& &\hspace{1cm} \leq  \left(\Hat{Y}_{m,m,a}  (\xi^2(t,s;x,y);x,y)\right)^{\frac{1}{2}}   \times \left(\Hat{Y}_{m,m,a} \left(\left(t- \frac{x\log(a)}{ \left( a^{\frac{1}{m}}-1\right)m} \right)^2\left((t-x)^2+(s-y)^2\right)\right);x,y\right)^{\frac{1}{2}}.
\end{eqnarray}
Using Matheamatica software, we get following results
\begin{eqnarray}\label{E12}
\nonumber\Hat{Y}_{m,m,a}\left(\left(t- \frac{x\log(a)}{ \left( a^{\frac{1}{m}}-1\right)m} \right)^2(t-x)^2;x,y\right) & = &\frac{1}{\left(a^{\frac{1}{m}}-1\right)^3m^4} \Bigg( x \log (a) \Bigg( \left(a^{\frac{1}{m}}-1\right)^2 (m x-1)^2\\
&&-x \left(a^{\frac{1}{m}}-1\right) \log (a) (2 m x-5)+x^2 (\log (a))^2\Bigg)\Bigg),
\end{eqnarray}
also, 
\begin{eqnarray}\label{E13}
\nonumber\Hat{Y}_{m,m,a}\left(\left(t- \frac{x\log(a)}{ \left( a^{\frac{1}{m}}-1\right)m} \right)^2(s-y)^2;x,y\right) & = & \frac{x y \log (a)}{m^4 \left(a^{\frac{1}{m}}-1\right)^3}\Bigg(m^2 y \left(a^{\frac{1}{m}}-1\right)^2\\
&&-\left(a^{\frac{1}{m}}-1\right) \log (a) (2 m y-1)+y \log ^2(a) \Bigg),
\end{eqnarray}
where $$\underset{m\to\infty}\lim\frac{\log(a)}{ \left( a^{\frac{1}{m}}-1\right)m}=1.$$
Now,  by simple calculations, we get
\begin{eqnarray}\label{E14}
\nonumber\underset{m\to\infty}\lim m^2\Hat{Y}_{m,m,a}\left(\left(t- \frac{x\log(a)}{ \left( a^{\frac{1}{m}}-1\right)m} \right)^2(t-x)^2;x,y\right) & = & \underset{m\to\infty}\lim \frac{1}{\left(a^{\frac{1}{m}}-1\right)^3m^2} \Bigg( x \log (a) \Bigg( \left(a^{\frac{1}{m}}-1\right)^2 (m x-1)^2\\\nonumber
&&-x \left(a^{\frac{1}{m}}-1\right) \log (a) (2 m x-5)+x^2 (\log (a))^2\Bigg)\Bigg)\\
&=& 3x^2,
\end{eqnarray}
also,
\begin{eqnarray}\label{E15}
\nonumber\underset{m\to\infty}\lim m^2\Hat{Y}_{m,m,a}\left(\left(t- \frac{x\log(a)}{ \left( a^{\frac{1}{m}}-1\right)m} \right)^2(s-y)^2;x,y\right) & = & \underset{m\to\infty}\lim \Bigg(\frac{x y \log (a)}{m^4 \left(a^{\frac{1}{m}}-1\right)^3}\Bigg(m^2 y \left(a^{\frac{1}{m}}-1\right)^2\\\nonumber
&&-\left(a^{\frac{1}{m}}-1\right) \log (a) (2 m y-1)+y \log ^2(a) \Bigg) \Bigg)\\
&=& xy,
\end{eqnarray}
here, we have 
\begin{eqnarray}\label{E16}
\underset{n\to}\lim \Hat{Y}_{m,m,a}  (\xi^2(t,s;x,y);x,y)=\xi^2(x,y))=0,
\end{eqnarray}
using (\ref{E10}-\ref{E16}) in \ref{E9}, which produce as
\begin{eqnarray*}
\underset{m\to\infty}{\lim}\frac{\partial }{\partial x}(\Hat{Y}_{m,m,a}(f;x,y))=\frac{\partial f(x,y)}{\partial x},
\end{eqnarray*} 
by same process, one can obtain the second result, which as follows:
\begin{eqnarray*}
\underset{m\to\infty}{\lim}\frac{\partial }{\partial y}(\Hat{Y}_{m,m,a}(f;x,y))=\frac{\partial f(x,y)}{\partial y}.
\end{eqnarray*}
Thus, the proof is completed.
\end{proof}
\pagebreak

\begin{example}
Consider a function $g(x,y)=x^2e^{x+y}$(green), we shall observe that the partial derivative of the bivariate operators \ref{R1} converge to the partial derivative of the function $g(x,y)$ by graphical representation. By taking $m=n=10,~30$ and corresponding derivative operators $\frac{\partial }{\partial x}\Hat{Y}_{10,10,a}(f;x,y)$ (red), $\frac{\partial }{\partial x}\Hat{Y}_{30,30,a}(f;x,y)$ (blue) and the graphical representation is shown in Figure \ref{F6}. 
 \linebreak
 \begin{figure}[h!]
    \centering 
    \includegraphics[width=.52\textwidth]{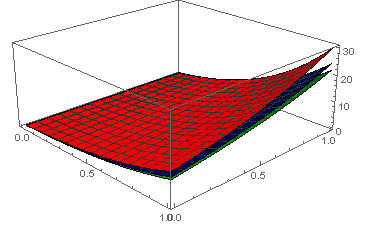}   
    \caption[Description in LOF, taken from~\cite{source}]{The convergence of the partial derivative of the operators $\Hat{Y}_{m,n,a}(f;x,y)$ with respect to $x$ to the partial derivative of the function $f(x,y)$(green)}
    \label{F6}
\end{figure}
\end{example}

\textbf{Concluding remark:} The above Figure \ref{F6}, represents the convergence of the partial derivative of the bivariate operators defined by \ref{R1}, with respect to $x$ to the partial derivative of the function w.r.t. $x$, i.e. for the large value of the $m,n$, the partial derivative of the bivariate operators converge to the partial derivative of the function. By the above Figure \ref{F6}, the error of the operators to the function will be very least when the value of $m,n$ is increased.

\section{Tabular and graphical representation}\label{se6}
In this section, it will be determined the approximation by the bivariate operators (\ref{R1}) for the function defined on $X=[0,\infty)\times[0,\infty)$ with the help of tables as well as graphical representation for different values of $m=n$.  
\begin{example}
By considering $\Hat{A}_{m,n,a}f=|\Hat{Y}_{m,n,a}(f;x,y)-f(x,y)|$, we find the numerical error by defined operators of the function $f(x,y)=xy$ at different points for different values of $n=m$. And also we can see, the error is least when the value of $m$ is increased, i.e. for the large value of $m,~n$, the bivariate operators will converge to the function. The numerical errors are given below, and illustrated in Table \ref{r5}.
\end{example}
\begin{center}
\begin{table}[h!]
\caption{Error bound of the operators $\Hat{Y}_{m,m,a}(f;x,y)$ to the function $f(x,y)$}\label{r5}
\begin{scriptsize}
\begin{tabular}{|c|c|c|c|c|c|c|c|c|c|}
\hline 
(x,y) & $\Hat{A}_{5,5,a}f$ & $\Hat{A}_{15,15,a}f$ & $\Hat{A}_{35,35,a}f$ & $\Hat{A}_{50,50,a}f$ & $\Hat{A}_{80,80,a}f$ & $\Hat{A}_{120,120,a}f$ & $\Hat{A}_{275,275,a}f$ & $\Hat{A}_{350,350,a}f$ & $\Hat{A}_{500,n,a}f$ \\ 
\hline 
(.1,.1) & .000542559 & .000183666 & .0000790616 & .000055398 & .0000346538 & .0000231137 & .0000100915 & 7.92971$\times 10^{-6}$  & 5.55135$\times 10^{-6}$ \\ 
\hline 
(.1,.3) & .00162768 & .000550997 & .000237185 & .000166194 & .000103961 & .0000693411 & .0000302744 & .0000237891 & .0000166541 \\ 
\hline 
(.1,.6) & .00325536 & .00110199 & .00047437 & .000332388 & .000207923 & .000138682 & .0000605487 & .0000475783 & .0000333081 \\ 
\hline 
(.2,.1) & .00108512 & .000367331 & .000158123 & .000110796 & .0000693076 & .0000462274 & .0000201829 & .0000158594 & .0000111027 \\ 
\hline 
(.2,.4) & .00434047 & .00146933 & .000632493 & .000443184 & .000277231 & .000184909 & .0000807316 & .0000634377 & .0000444108 \\ 
\hline 
(.3,.4) & .00651071 & .00220399 & .000948739 & .000664776 & .000415846 & .000277364 & .000121097 & .0000951566 & .0000666162 \\ 
\hline 
(.6,.2) & .00651071 & .00220399 & .000948739 & .000664776 & .000415846 & .000277364 & .000121097 & .0000951566 & .0000666162 \\ 
\hline 
(.7,.8) & .0303833 & .0102853 & .00442745 & .00310229 & .00194061 & .00129437 & .000565121 & .000444064 & .000310876 \\ 
\hline 
(.8,.7) & .0303833 & .0102853 & .00442745 & .00310229 & .00194061 & .00129437 & .000565121 & .000444064 & .000310876 \\ 
\hline 
(.9,.3) & .0146491 & .00495897 & .00213466 & .00149575 & .000935653 & .00062407 & .000272469 & .000214102 & .000149886 \\ 
\hline 
(.9,.9) & .0439473 & .0148769 & .00640399 & .00448724 & .00280696 & .00187221 & .000817408 & .000642307 & .000449659 \\ 
\hline 
\end{tabular} 
\end{scriptsize}
\end{table}
\end{center}

\begin{example}
Here function is $f(x,y)=x^4y(x-1)^4\sin{2\pi y}$ (green), choosing $n=m=10,~50,~100$ and fixed $a=1.32$, then corresponding operators are $\Hat{Y}_{10,10,a}(f;x,y)$ (yellow), $\Hat{Y}_{50,50,a}(f;x,y)$ (red), $\Hat{Y}_{100,100,a}(f;x,y)$ (blue),  we can see that as we increase the value of $m,n$, the convergence becomes better.  The representation is illustrated in Figure \ref{F5}.

\begin{figure}[h!]
    \centering 
    \includegraphics[width=.82\textwidth]{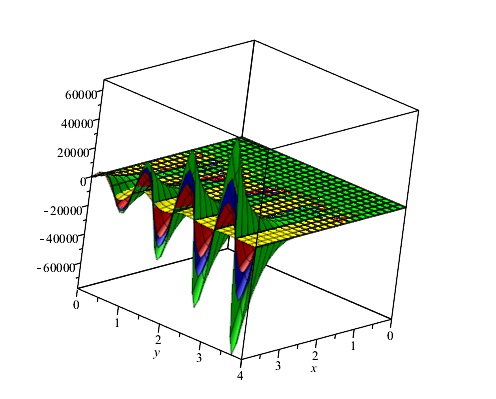}   
    \caption[Description in LOF, taken from~\cite{source}]{The convergence of the bivariate operators $\Hat{Y}_{m,n,a}(f;x,y)$ to the function $f(x,y)$.}
    \label{F5}
\end{figure}
\end{example}

\textbf{Graphical conclusion:} By Figure \ref{F5}, one can observe that the better approximation takes place when the values of $m,n$ (while $m=n$) are increased. After the certain value of $m,n$, the operators converge to the function. 


\begin{example}
Taking $f(x,y)=x^3y + 6y^2 + x^2$, for different values of $m,n$, we get the numerical error of bivariate operators \ref{R1} and \ref{S} to the function $f(x,y)$. We can observe that the numerical error is least comparison to bivariate Sz$\acute{\text{a}}$sz-Mirakjan operators, for fixed $a=1.30$, and is illustrated by given Table \ref{r6}.
\begin{table}[h!]
\caption{The comparison of the bivariate operators  $\Hat{Y}_{m,n,a}(f;x,y)$ and $S_{m,n,a}(f;x,y)$ operators}\label{r6}
\begin{tabular}{ |p{.75cm}|p{3.8cm} |p{3.4cm}|}
\hline
\multicolumn{3}{|c|}{The comparison of the bivariate operators  $\Hat{Y}_{m,n,a}(f;x,y)$ and $S_{m,n,a}(f;x,y)$ operators } \\
\hline
m, n & $|{\Hat{Y}_{m,n,a}(f;x,y)-f(x,y)}|$  & $|{S_{m,n}(f;x,y)-f(x,y)}|$ \\
\hline
3 & 0.219254 & 0.235444  \\
\hline
 13 & 0.0521816 & 0.0541361  \\
\hline
 23 & 0.0296115 & 0.0305841  \\
\hline
 33 & 0.0206706 & 0.0213122  \\
\hline
 53 & 0.0128878 & 0.0132677  \\
\hline
 103 & 0.00663879 &  0.00682619  \\
\hline
 163 & 0.00419684 &  0.00431326  \\
\hline
 333 & 0.00205507 &  0.0021112  \\
\hline
 543 & 0.00126047 & 0.00129469  \\
\hline
\end{tabular}
\end{table}
\end{example}
\begin{example}
This example contains the comparison of the proposed bivariate \ref{R1}, for a fixed value of $a=1.32$ with bivariate Sz$\acute{\text{a}}$sz-Mirakjan operators \ref{S} by graphical representation, while having numerical value $10$ for $m=n$ and choosing function $f(x,y)=y^2e^{2x}$, we get the  better rate of convergence, is illustrated by Figure \ref{F7}.
\begin{figure}[h!]
    \centering 
    \includegraphics[width=.50\textwidth]{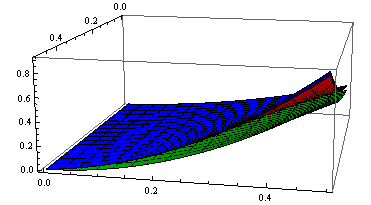}   
    \caption[Description in LOF, taken from~\cite{source}]{The comparison of the proposed bivariate operators  $\Hat{Y}_{10,10,a}(f;x,y)$ (red) and $S_{10,10}(f;x,y)$ (blue) to the function $f(x,y)$ (green)}
    \label{F7}
\end{figure}
\end{example}
\pagebreak
\textbf{Concluding Remark:} By above Figure \ref{F7}, we observe that  the approach of the proposed bivariate operators $\Hat{Y}_{m,n,a}(f;x,y)$  is better than the operators $S_{m,n}(f;x,y)$.

\section{Application}
As applications of the defined operators (\ref{R1}), one can generalize into complex extension of the said operators for the function $f:[0,\infty)\times[0,\infty)\to \mathbb{C}\times\mathbb{C}$, while our operators are evidence in the real case. In the real case, one can define the Stancu variant for function of two variables and find the better rate of the convergence as well fast approximation result. Better approximation can be obtained and approximation results are discussed in various field like wavelets and frame, number theory, theoretical physics as well as in the field of quantum calculus by modifying the said operators with $q$-integers.

\end{document}